\documentclass{amsart}
\usepackage[utf8]{inputenc}
\usepackage{amsmath, amssymb}
\usepackage{amsthm}
\usepackage[all]{xy}
\usepackage{color}
\definecolor{darkblue}{rgb}{0,0,0.6}
\usepackage[ocgcolorlinks,colorlinks=true, citecolor=darkblue, filecolor=darkblue,
linkcolor=darkblue, urlcolor=darkblue]{hyperref}
\usepackage[nameinlink,capitalise,noabbrev]{cleveref}
\usepackage{bbm}
\usepackage{colonequals}  
\usepackage{enumerate}
\usepackage{verbatim}  

\makeatletter
\@namedef{subjclassname@2010}{%
  \textup{2010} Mathematics Subject Classification}
\makeatother
\title[Metric transforms yielding Gromov hyperbolic spaces]{Metric transforms yielding Gromov hyperbolic spaces}
\author[Dragomir]{George Dragomir}
\address{Department of Mathematics and Statistics, McMaster University, Hamilton, Ontario, Canada L8S 4K1}
\email{dragomir@math.mcmaster.ca}

 \author[Nicas]{Andrew Nicas}
\address{Department of Mathematics and Statistics, McMaster University, Hamilton, Ontario, Canada L8S 4K1}
\email{nicas@mcmaster.ca}
\thanks{The second author was partially supported by a grant from
the Natural Sciences and Engineering Research Council of Canada. The authors would also like to thank the referee for many valuable comments.}


\date{\today}

\newcommand{\IR}{\mathbb{R}}

\newcommand{\bbR}{\mathbbm{R}}

\numberwithin{equation}{section}

\newtheorem{theorem}[equation]{Theorem}

\newtheorem{proposition}[equation]{Proposition}

\newtheorem{lemma}[equation]{Lemma}
\newtheorem{thmA}{Theorem}

\newtheorem{cornono}{Corollary}

\newtheorem{cor+}{Corollary}
\newtheorem{prop+}[cor+]{Proposition}
\theoremstyle{definition}

\newtheorem{example}[equation]{Example}

\newtheorem*{definition}{Definition}

\newtheorem*{corollary}{Corollary}

\newtheorem{remark}[equation]{Remark}

 \newtheoremstyle{TheoremNum}
        {}{}              
        {\itshape}                      
        {}                              
        {\bfseries}                     
        {.}                             
        { }                             
        {\normalfont\bfseries\thmname{#1}\thmnote{#3}}
\theoremstyle{TheoremNum}

\calclayout

\begin{document}

\begin{abstract}
A real valued function $\varphi$ of one variable  is called a metric transform if for every metric space $(X,d)$ the composition $d_\varphi = \varphi\circ d$ is also a metric on $X$. 
We give a complete characterization of the class of approximately nondecreasing, unbounded metric transforms $\varphi$ such that the transformed Euclidean half line $([0,\infty),|\cdot|_\varphi)$ is Gromov hyperbolic. 
As a consequence, we obtain \emph{metric transform rigidity} for roughly geodesic Gromov hyperbolic spaces, that is, if $(X,d)$ is any metric space containing a rough geodesic ray and $\varphi$ is an approximately nondecreasing, unbounded metric transform such that the transformed space $(X,d_\varphi)$ is Gromov hyperbolic and roughly geodesic then $\varphi$ is an approximate dilation and the original space $(X,d)$ is Gromov hyperbolic and roughly geodesic.
\end{abstract}

\maketitle

\baselineskip 15pt

\section{Introduction}\label{intro}

A function $\varphi\colon [0,\infty)\to[0,\infty)$ is called a {\it metric transform} if for each metric space $(X,d)$ the composition $d_\varphi=\varphi\circ d$ is also a metric on $X$.
A metric transform $\varphi$ is necessarily subadditive and satisfies $\varphi^{-1}(0)=\{0\}$.
While  these two conditions on $\varphi$ are not sufficient for it to be a metric transform, if we further require that $\varphi$ is nondecreasing then it is a metric transform.
In particular, any nonconstant, nonnegative concave function $\varphi$  with domain $[0,\infty)$ and satisfying $\varphi(0)=0$ is a metric transform.

A central question concerning metric transforms is whether there exist metric transforms $\varphi$ for which the transformed metric space $(X,d_\varphi)$ has certain specified properties or preserves some of the characteristics of the original metric space $(X,d)$.
Early results about transformed metric spaces dealt with their ``Euclidean" properties.
Blumenthal \cite{Blu:43Rem:aa} showed that if $0<\alpha\le \tfrac{1}{2}$ and $(X,d)$ is  any metric space then the snowflake metric $d^\alpha$ has the property that any four points of $(X,d^\alpha)$ can be isometrically embedded into Euclidean space. 
Wilson \cite{Wil:35On-:aa} showed that the real line with the snowflake metric $|t-s|^{1/2}$ embeds isometrically in 
a real Hilbert space, but cannot embed isometrically in any finite dimensional Euclidean space. 
Remarkable results in this direction were obtained by Schoenberg, independently in \cite{Sch:38Met:ab,Sch:38Met:aa} and, together with von Neumann in \cite{Neu:41Fou:aa}, where they determined all metric transforms for which a transformed Euclidean space isometrically embeds into another Euclidean space.  See \cite[Chapter 9]{Deza-Laurent}) for a discussion and \cite{Le-DRW} for some recent developments.

Our aim is to investigate analogous types of questions in the context of Gromov hyperbolic spaces. 
Recall that if $(X,d)$ is a metric space and $x,y,w\in X$\, then the {\it Gromov product} of $x$ and $y$ with respect to $w$ is defined as \begin{equation*}(x\mid y)_w = \tfrac{1}{2}\left[d(x,w)+d(y,w)-d(x,y)\right].\end{equation*} 
Given $\delta\ge 0$, the metric space $(X,d)$ is said to be {\it $\delta$-hyperbolic} if \[(x\mid y)_w\ge\min\left\{(x\mid z)_w, (y\mid z)_w\right\}-\delta\] for all $x,y,z,w\in X$. A metric space $(X,d)$ is said to be {\it Gromov hyperbolic} if it is $\delta$-hyperbolic for some $\delta\ge0$. 

A basic example of a Gromov hyperbolic metric space is $([0,\infty),|\cdot|)$, the half line with the \hbox{Euclidean} metric. 
In this case, the Gromov product based at $0$ is $(t \, \mid s)_0=\min\{t,s\}$ 
and the space is \hbox{$0$-hyperbolic.} More generally, any $\IR$-tree is $0$-hyperbolic. 
Another well-known example is the hyperbolic plane, which is $\log(3)$-hyperbolic. 
A Euclidean space of dimension greater than $1$ is not Gromov hyperbolic. 
While Gromov hyperbolicity is a quasi-isometry invariant for {\it intrinsic} metric spaces \cite[Theorems 3.18 and 3.20]{Vai:05Gro:aa},
quasi-isometry invariance can fail for non-intrinsic spaces, see \cite[Remark 3.19]{Vai:05Gro:aa} and also \cite[Remarque 13, p.89]{Ghys-dlH}.

We say that a function is {\it approximately nondecreasing} if it is within bounded distance from a nondecreasing function.
Our first result gives a complete characterization of the class of approximately nondecreasing, unbounded metric transforms $\varphi$ such that $([0,\infty),|\cdot|_\varphi)$ is Gromov hyperbolic.
Some additional terminology will be useful.
Recall that a {\it dilation} on $[0,\infty)$ is a function of the form $t \mapsto \lambda t$ where $\lambda$ is a positive constant.
We say that the function $\varphi$ is an {\it approximate dilation} if it is within bounded distance from a dilation.
Furthermore, we say that $\varphi$  is {\it logarithm-like}  if the function
$t ~\mapsto~ \varphi(2t)-\varphi(t)$  is bounded from above.

\begin{thmA}\label{thm:A}
Let $\varphi$ be an approximately nondecreasing, unbounded metric transform.
The \hbox{transformed} metric space $([0,\infty),|\cdot|_\varphi)$ is Gromov hyperbolic if and only if one of the following two mutually \hbox{exclusive} conditions holds:
\begin{itemize}
\item[$(i)$]  $\varphi$ is an approximate dilation, or
\item[$(ii)$] $\varphi$ is logarithm-like.
\end{itemize}
\end{thmA}

It is straightforward to show that if $\varphi$ is a metric transform and also an approximate dilation then $\varphi$ preserves Gromov hyperbolicity, that is, if $(X,d)$ is any Gromov hyperbolic space then $(X,d_{\varphi})$ is also Gromov hyperbolic (Proposition~\ref{prop:appr_dilation}). 
 
 The function $t \mapsto \log(1+t)$ is a metric transform and logarithm-like (as defined above), indeed the inspiration for the terminology ``logarithm-like''.
 Gromov observed that if $(X,d)$ is {\it any} metric space then $(X,\,\log(1+d))$ is Gromov hyperbolic (\cite[Example 1.2(c)]{Gro:87Hyp:aa}).
 More generally, we show that if an approximately nondecreasing metric transform $\varphi$ is logarithm-like then the transformed space $(X, d_{\varphi})$ is ``approximately ultrametric'' and hence Gromov hyperbolic  (Proposition~\ref{prop:appr_log-like}).
We say that a metric space $(X,d)$ is {\it approximately ultrametric} if there exists $\delta \geq 0$ such that for all $x, y, z \in X$ the inequality $d(x,y) \leq  \max\left\{ d(x,z),  d(z,y)\right\} + \delta$ is satisfied. 
An unbounded, approximately ultrametric space fails to have the rough midpoint property and so is never a rough geodesic metric space (Proposition~\ref{prop:appr_ultra_non_geo}).

A {\it rough geodesic ray} in a metric space $(X,d)$ is a rough isometric embedding of the Euclidean half line in $X$,
that is, a function $\gamma\colon [0,\infty)\to X$ and a constant $k\ge 0$ such that for all $t,s\ge 0$, $|t-s|-k\le d(\gamma(t),\gamma(s))\le |t-s|+k$.

Theorem~\ref{thm:A} has the following consequence.

\begin{thmA}\label{thm:B}
Let $(X,d)$ be a metric space containing a rough geodesic ray.
Let $\varphi$ be an approximately nondecreasing, unbounded metric transform.
If  the transformed space $(X,d_\varphi)$ is Gromov hyperbolic then
\begin{itemize}
\item[$(i)$] $(X,d)$ is Gromov hyperbolic and $\varphi$ is an approximate dilation, or
\item[$(ii)$] $(X,d_\varphi)$ is approximately ultrametric. \\
{Conditions (i) and (ii) are mutually exclusive.}                    
\end{itemize}
\end{thmA}

Since an unbounded, approximately ultrametric space is never roughly geodesic, Theorem~\ref{thm:B} \linebreak 
immediately yields the following corollary which can be viewed as a type of  rigidity with respect to metric transformation of roughly geodesic Gromov hyperbolic spaces.

\begin{cornono}[Metric Transform Rigidity]
Let $(X,d)$ be a metric space containing a rough geodesic ray.
Let $\varphi$ be an approximately nondecreasing, unbounded metric transform.
If  the transformed space $(X,d_\varphi)$ is Gromov hyperbolic and roughly geodesic then $\varphi$ is an approximate dilation and $(X,d)$ is Gromov hyperbolic and roughly geodesic.
\end{cornono}

This paper is organized as follows. In Section~\ref{sec:metric_transforms} we recall  some of the relevant properties of metric transforms and concave functions.
In Section~\ref{sec:Gromov_hyperbolic}, after reviewing some useful facts concerning Gromov hyperbolic spaces, we introduce approximately ultrametric spaces and discuss some of their immediate properties.
In Section~\ref{sec:conc_mt} we give a complete characterization of all concave functions that transform the Euclidean half line into a Gromov hyperbolic space (Theorem~\ref{thm:conc_mainA}).
We extend this result to the case of approximately nondecreasing,  unbounded metric transforms in Section~\ref{sec:app_conc_mt}, where we prove Theorem~\ref{thm:A}.
The proof of Theorem~\ref{thm:B} and  its application to  roughly geodesic Gromov hyperbolic spaces is given in Section~\ref{sec:proof_thmB}.


\section{Metric Transforms and Approximately Concave Functions}\label{sec:metric_transforms}
We summarize some properties of metric transforms, concave functions and approximately concave functions that will be needed in the sequel.

\subsection{Metric transforms}
\label{subsec:mt}

General treatments of metric transforms can be found in
\cite{Cor:99Int:aa,Deza-Laurent}. 
Translation invariant distances on the real line are studied in \cite{Le-D}.

\begin{definition}
A function $\varphi\colon [0,\infty)\to [0,\infty)$ is said to be a \emph{metric transform} if for every metric space $(X,d)$ the space $(X,d_\varphi)$ with $d_\varphi(x,y) = \varphi(d(x,y))$ is again a metric space. We denote by $\mathcal{M}$ the class of all metric transforms.
\end{definition}

For any $\varphi\in\mathcal{M}$, since $d_\varphi(x,y)=0$ if and only if $x=y$, we have that $\varphi(t)=0$ if and only if $t=0$.
Hence, a necessary condition for a function $\varphi\colon [0,\infty)\to[0,\infty)$ to be a metric transform is that $\varphi^{-1}(0)=\{0\}$. 

A complete, albeit somewhat tautological, characterization of the elements of $\mathcal{M}$ can be given as follows. A triplet $(a,b,c)$ of nonnegative real numbers is called a {\it triangle triplet} if $a\le b+c,$  $b\le a+c \mbox{ and } c\le a+b.$ 

\begin{proposition}[{\cite[2.6]{Cor:99Int:aa}}]\label{prop:mt-charact}
Assume $\varphi\colon [0,\infty)\to[0,\infty)$ satisfies $\varphi^{-1}(0)=\{0\}$. Then $\varphi$ is a metric transform if and only if $(\varphi(a),\varphi(b),\varphi(c))$ is a triangle triple whenever $(a,b,c)$ is one. \qed
\end{proposition}

Proposition~\ref{prop:mt-charact} implies the following properties of metric transforms.

\begin{proposition}\label{prop:mt-properties}
Assume $\varphi\in \mathcal{M}$. Then
\begin{itemize}
\item[$(i)$] $\varphi$ is subadditive, that is, $\varphi(t+s)\le\varphi(t)+\varphi(s)$, for all $t,s\ge0$,
\item[$(ii)$] $|\varphi(t)-\varphi(s)|\le \varphi(|t-s|)$, for all $t,s\ge0.$\qed
\end{itemize}
\end{proposition}

While subadditivity and $\varphi^{-1}(0)=\{0\}$ are necessary conditions for a function $\varphi\colon [0,\infty)\to[0,\infty)$ to be a metric transform, these conditions are, in general, not sufficient (see  Example~\ref{ex:non_mt}).
However, if  $\varphi$ is also nondecreasing then it follows  from Proposition~\ref{prop:mt-charact} that $\varphi\in\mathcal{M}$.
We summarize this as follows.

\begin{proposition}[{\cite[2.3]{Cor:99Int:aa}}]\label{prop:subadd+inc=mt}
Assume $\varphi\colon [0,\infty)\to[0,\infty)$ with $\varphi^{-1}(0)=\{0\}$ is subadditive and nondecreasing.
Then $\varphi$ is a metric transform.  \qed
\end{proposition}

\begin{example}\label{ex:non_mt}
Let $\varphi\colon [0,\infty)\to[0,\infty)$ be given by $ \varphi(t)=at+b|\sin(t)|$ with $a>0$ and $b\ge0$.
Then $\varphi^{-1}(0)=\{0\}$, and the subadditivity of $\varphi$ follows from $|\sin(t+s)|=|\sin(t)\cos(s)+\sin(s)\cos(t)|\le |\sin(t)|+|\sin(s)|.$  
Note that if $a\ge b$\, then $\varphi$ is nondecreasing and, by Proposition~\ref{prop:subadd+inc=mt}, $\varphi$ is a metric transform.
Also note that $\varphi$ is not concave unless $b=0$. 
If $a<b$\, then $\varphi$ is not monotonic and not a metric transform.
\end{example}

\begin{remark}
In general, metric transforms need not be continuous. It follows from  part $(ii)$ of Proposition~\ref{prop:mt-properties} that if $\varphi\in \mathcal{M}$ is continuous at $0$ from the right then $\varphi$ is continuous on $[0,\infty)$. Furthermore, a transformed space $(X,d_\varphi)$ is topologically equivalent to the original space $(X,d)$ if and only if $\varphi$ is continuous.
The metric topology on $(X,d_\varphi)$ is discrete for every metric $d$\, if and only if  $\varphi$ is discontinuous at $0$ (see {\cite[3.1]{Cor:99Int:aa}}). 
Similarly, the differentiability of a metric transform is influenced by its behaviour near $0$. 
For any metric transform $\varphi$, the right derivative $\varphi'_+(0)$ exists in the extended sense
 (we allow infinite values)
and if $\varphi_+'(0)<\infty$\, then $\varphi$ is $\varphi_+'(0)$-Lipschitz on $(0,\infty)$ and therefore differentiable  except possibly at countably many points (see  {\cite[4.7]{Cor:99Int:aa}}).
\end{remark}

Throughout this paper, unless otherwise specified, metric transforms are not assumed to be continuous. 


\subsection{Concave functions}\label{subsec:concave}

In this subsection,
after a very brief review of some basic properties of concave functions,
we summarize some results concerning continuous concave functions $\varphi\colon[0,\infty) \to [0,\infty)$ satisfying $\varphi(0)=0$ that will be used in Section~\ref{sec:conc_mt}.

 Let $\varphi\colon I\to\IR$ be defined on some interval $I\subseteq\IR$, that is, a connected subset of $\IR$.
 The function $\varphi$ is \emph{concave} if for all $x,y\in I$ and all $t\in[0,1]$, \begin{equation*}  (1-t)\varphi(x)+t\varphi(y)\le \varphi((1-t)x+ty).\end{equation*}
Reversing the above inequality gives the definition of a \emph{convex} function.
Hence, $\varphi$ is concave if and only if $-\varphi$ is convex.

Convex functions have been extensively studied and many of their properties are well known. 
We recall some properties of concave functions that we need, omitting the proofs as these can be found, for instance, in \cite[Chapter I]{Rob:73Con:aa}.  

By definition, a function $\varphi$ is concave if and only if any portion of its graph lies on or above the chord connecting the end points of this portion of the graph.
Alternatively, $\varphi$ is concave if and only if any of the following inequalities
\[
\frac{\varphi(z)-\varphi(x)}{z-x} \ge \frac{\varphi(y)-\varphi(x)}{y-x}\ge \frac{\varphi(y)-\varphi(z)}{y-z}
\]
 hold for all $x<z<y$ (see \cite[Sec. I.10 (2)]{Rob:73Con:aa}).

The following elementary properties  of concave functions ({see \cite[Theorems I.10.A, B and C]{Rob:73Con:aa}}) will be useful.

\begin{proposition}\label{prop:conc_diff}
Assume $\varphi\colon [0,\infty)\to \IR$ is a concave function.
Then $\varphi$ satisfies a Lipschitz \hbox{condition} on any compact interval contained in $(0,\infty)$ and is therefore  continuous on $(0,\infty)$. The left derivative $\varphi'_-$ and the right derivative $\varphi'_+$ of $\varphi$ exist at every point in $(0,\infty)$ and \[\varphi'_-(x)\ge\varphi'_+(x)\ge \frac{\varphi(y)-\varphi(x)}{y-x}\ge \varphi'_-(y)\ge \varphi'_+(y)\] for all $0<x<y$.
Hence both one-sided derivatives are nonincreasing on $(0,\infty)$ and, at each point, the left derivative is no smaller than the right derivative.
In particular, $\varphi$ is differentiable on $(0,\infty)$ except possibly at countably many points. \qed
\end{proposition}

The next two corollaries are direct consequences of Proposition~\ref{prop:conc_diff}. 
The first is a version of the Fundamental Theorem of Calculus (see Remark I.12.B in \cite{Rob:73Con:aa}), and the second uses the fact that any local maximum of a concave function is also a global maximum. 
Together with the assumption that $\varphi$ is nonnegative, this implies that if $\varphi$ is unbounded then it must  be increasing.  

\begin{corollary}\label{cor:conc_FT}
If $\varphi\colon [0,\infty)\to\IR$ is a continuous concave function then for any $x,y \ge 0$,
\[
 \qquad  \qquad \qquad \qquad \qquad \qquad  \varphi(y)-\varphi(x) = \int_x^y\varphi'_-(t)dt = \int_x^y\varphi'_+(t)dt. \qquad \qquad \qquad \qquad \qquad \qquad  \qed  
\]  
\end{corollary}

\begin{corollary}\label{cor:phi_global}
Let $\varphi\colon [0,\infty)\to[0,\infty)$ be a concave function such that $\varphi(0)=0$ and $\varphi$ is not constant on $(0,\infty)$. Then either of the following holds:
\begin{itemize}
\item[$(i)$] $\varphi$ is strictly increasing, or
\item[$(ii)$] there exists $a> 0$ such that $\varphi$ is strictly increasing on $[0,a)$ and  constant on $[a,\infty)$.
 \qed
\end{itemize}
\end{corollary}

\begin{remark}\label{rem:phi'_lim}
Since the one-sided derivatives are nonincreasing and $\varphi'_-(x)\ge \varphi'_+(x)\ge 0$ for all $x>0$, there exists $\lambda\ge 0$ such that 
$$ \lim_{x\to\infty}\varphi'_-(x) = \lim_{x\to\infty}\varphi'_+(x)=\lambda.$$
\end{remark}

Another important property of concave functions defined on a possibly infinite open interval $I\subseteq \IR$ is that they can be globally approximated by concave functions which are real analytic on $I$.
Azagra showed  (\cite[Theorem 1.1]{Aza:13Glo:aa}) that for every convex function $f\colon U\to \IR$ defined on an open convex subset $U\subseteq\IR^n$, $n\ge 1$, and every $\varepsilon>0$, there exists a real analytic convex function $g\colon U\to\IR$ such that $f-\varepsilon\le g\le f$.
We will only be interested in uniform approximations by functions which are of class $C^1$ and
so the following weaker version of Azagra's theorem for concave functions is sufficient for our purpose.

\begin{proposition}[{\cite[Theorem 1.1]{Aza:13Glo:aa}}]\label{prop:conc_approx_open}
Let $U\subseteq \IR$ be an open connected subset and let $\varphi\colon U\to\IR$  be a concave function. For every $\varepsilon>0$ there exists a $C^1$ concave function $\sigma\colon U\to\IR$ such that $\varphi\le \sigma\le \varphi+\varepsilon$. \qed
\end{proposition}

If $\varphi\colon [0,\infty)\to[0,\infty)$ is a continuous concave function such that $\varphi(0)=0$ then Proposition~\ref{prop:conc_approx_open} provides $C^1$ concave approximations of $\varphi$ on the open interval $(0,\infty)$. 
To obtain an approximation of $\varphi$ by a concave function which is continuous on $[0,\infty)$ and of class $C^1$ on $(0,\infty)$, we wish to extend $\varphi$ to a continuous concave function defined at the left of $0$.
Observe that this is not possible if the right derivative $\varphi'_+(0)$ is undefined, but as the proof of the next result shows, we can slightly modify the function $\varphi$ near zero in order for such a continuous concave extension to exist. 
The following corollary plays a key role in the proof of Proposition~\ref{prop:w_hat}.
\begin{corollary}\label{cor:conc_approx}
Let $\varphi\colon [0,\infty)\to[0,\infty)$ be a continuous concave function such that $\varphi(0)=0$. For every $\varepsilon>0$ there exists a continuous concave function $\psi\colon [0,\infty)\to[0,\infty)$ that 
 is of class $C^1$ on $(0,\infty)$, satisfies $\psi(0)=0$ and $|\varphi-\psi|\le\varepsilon$. 
\end{corollary}
\begin{proof}
Observe that, by concavity, if $a>0$ and $\varphi'_+(a)\le m \le \varphi'_-(a)$\, then $\varphi(x)\le\varphi(a)+m(x-a)$ for all $x\ge 0$,
that is, the graph of $\varphi$ lies on or under the line with slope $m$ and passing through $(a,\varphi(a))$.
Replacing the portion of the graph of $\varphi$ above $[0,a]$ by the line segment $m(x-a)+\varphi(a)$ we obtain a continuous concave function which extends indefinitely at the left of $0$ and which is equal to \hbox{$\varphi$ on $[a,\infty)$.} 

 Assume $\varphi$ is not identically 0. Let $\varepsilon>0$. By Corollary~\ref{cor:phi_global} and by possibly taking a smaller $\varepsilon>0$,
we may assume that $\varphi$ is strictly increasing on $\varphi^{-1}([0,\varepsilon])$.
Let $a=\varphi^{-1}(\varepsilon/2)$  
and $m=\varphi'_+(a)$, and define $\widetilde{\varphi}_\varepsilon\colon \IR\to\IR$ by
 \[
\widetilde{\varphi}_\varepsilon(x) = 
\begin{cases}
\varphi(x) & \mbox{ if }  x\ge\varphi^{-1}(\varepsilon/2)\\ \varphi'_+(\varphi^{-1}(\varepsilon/2)) [x-\varphi^{-1}(\varepsilon/2)] +\varepsilon/2  & \mbox{ if }  x<\varphi^{-1}(\varepsilon/2).
\end{cases}
\]
The function $\widetilde{\varphi}_\varepsilon$ is concave and for all $x\ge 0$ it satisfies $\varphi(x)\le \widetilde{\varphi}_\varepsilon(x)\le\varphi(x)+\varepsilon/2$.
By Proposition~\ref{prop:conc_approx_open}, applied to the function $\widetilde{\varphi}_\varepsilon\colon \IR\to\IR$ and with $\varepsilon = \varepsilon/2$, there exists a $C^1$ concave function $\widetilde{\sigma}\colon \IR\to\IR$ such that $\widetilde{\varphi}_\varepsilon\le\ \widetilde{\sigma}\le\widetilde{\varphi}_\varepsilon+\varepsilon/2.$
Furthermore, if $\varphi_\varepsilon = \widetilde{\varphi}_\varepsilon|_{[0,\infty)}$ and $\sigma = \widetilde{\sigma}|_{[0,\infty)}$\, then on $[0,\infty)$ we have 
$$0\le\varphi \le {\varphi}_\varepsilon \le {\sigma}\le {\varphi}_\varepsilon+\varepsilon/2\le \varphi+\varepsilon.$$
Since $0\le {\sigma}(0)\le \varepsilon$ and $\sigma$ is increasing, it follows that $\sigma(x)-\sigma(0)\ge 0$ and  $$\varphi(x)-\varepsilon\le \sigma(x)-\sigma(0)\le \varphi(x)+\varepsilon$$ for all $x\ge 0$. Define $\psi (x)=\sigma(x)-\sigma(0)$.
Then $\psi\colon [0,\infty)\to[0,\infty)$ is a continuous concave function which is $C^1$ on $(0,\infty)$ and satisfies $\psi(0)=0$ and $|\varphi-\psi|\le\varepsilon$.
\end{proof}

\subsection{Approximately concave functions}\label{subsec:appr_conc} 
In this subsection we show that approximately midpoint-concave functions can be uniformly approximated by continuous concave functions (Corollary~\ref{cor:appr-midconc_conc}).
This result will be used in Section ~\ref{sec:app_conc_mt} and Section~\ref{sec:proof_thmB}.

\begin{definition}\label{defn:e-conc}
Let $\varphi\colon I\to\IR$ be defined on some interval $I\subseteq\IR$, and let $\delta\ge0$.
\begin{itemize}
\item[$(i)$] $\varphi$ is said to be \emph{$\delta$-concave} if for all $x,y\in I$ and all $t\in[0,1]$, \begin{equation*} 
(1-t)\varphi(x)+t\varphi(y)\le\varphi((1-t)x+ty)+\delta.\end{equation*}
\item[$(ii)$] $\varphi$ is called \emph{$\delta$-midpoint-concave} (or {\it $\delta$-midconcave}) if for all $x,y\in I$, \begin{equation*}   
\tfrac{1}{2}\varphi(x)+  \tfrac{1}{2}\varphi(y)\le\varphi\left(\tfrac{x+y}{2}\right)+\delta.\end{equation*}
\end{itemize}
\end{definition}

We say that the function $\varphi$ is {\it approximately concave} (respectively, {\it approximately midpoint-concave}) if it is $\delta$-concave (respectively, $\delta$-midpoint-concave) for some 
$\delta \geq 0$.
Taking $\delta=0$ recovers the definition of a {\it concave} (respectively, {\it midpoint-concave}) function. 

If $\varphi$ is continuous (or locally bounded from below) then $\varphi$ is concave if and only if $\varphi$ is midpoint-concave (see  \cite[Theorem VII.71.C]{Rob:73Con:aa}). Here we show that approximately midpoint-concave functions $\varphi\colon [0,\infty)\to[0,\infty)$ with $\varphi(0)=0$ can be uniformly approximated on $[0,\infty)$ by continuous concave functions $\psi$ satisfying $\psi(0)=0$. This is a consequence of the following two results.

\begin{proposition}[{\cite[Theorem 1]{Ng:93On-:aa}}]\label{thm:e-midconc_2e-conc}
Let $I\subset \IR$ be an open interval.
If \hbox{$\varphi\colon I\to\IR$} is $\delta$-midpoint-concave and locally bounded from below at a point in $I$ then $\varphi$ is $2\delta$-concave.\qed
\end{proposition}

\begin{proposition}[{\cite[Theorem 2]{Hye:52App:aa}}]\label{thm:e-conc_conc}
Assume $\varphi\colon I\to\IR$ is $\delta$-concave on an open interval $I\subset\IR$.
Then there exists a continuous concave function $f\colon I\to \IR$ such that $|\varphi(x)-f(x)|\le \delta/2$,
\linebreak 
 for all $x\in I$. \qed
\end{proposition}

\begin{corollary}\label{cor:appr-midconc_conc}
Assume $\varphi\colon [0,\infty)\to [0,\infty)$ is approximately midpoint-concave and $\varphi(0)=0$.
Then there exists a continuous concave function $\psi\colon [0,\infty)\to[0,\infty)$ satisfying $\psi(0)=0$ and such that $|\varphi(x)-\psi(x)|$ is uniformly bounded on $[0,\infty)$.
\end{corollary}
\begin{proof}
Assume $\varphi\colon [0,\infty)\to[0,\infty)$ is $\delta$-midpoint-concave, for some $\delta\ge 0$. Since $\varphi$ is bounded from below by $0$ on $(0,\infty)$ by Proposition~\ref{thm:e-midconc_2e-conc}, $\varphi$ is $2\delta$-concave on $(0,\infty)$, and by Proposition~\ref{thm:e-conc_conc}, there exists a continuous concave function $f\colon (0,\infty)\to \IR$ such that $|\varphi(x)-f(x)|\le \delta$, for all $x>0$. Notice that $f$ is bounded from below by $ -\delta$ on $(0,\infty)$ and since $f$ is continuous, it is also nondecreasing (see Corollary~\ref{cor:phi_global}). Thus $f$ can be extended by continuity at $0$ and $f(0)=\lim_{x\to 0^+}f(x)\ge -\delta$. Define $\psi\colon [0,\infty)\to [0,\infty)$ by $\psi(x)=f(x)-f(0)$. Then $\psi$ is continuous, concave, and satisfies $\psi(0)=0$ and $|\varphi(x)-\psi(x)|\le |\varphi(x)-f(x)|+|f(0)|\le 2\delta$ for all $x\ge 0$. 
\end{proof}


\section{Gromov Hyperbolic Spaces}\label{sec:Gromov_hyperbolic}

 Gromov hyperbolic spaces were introduced by Gromov in his landmark paper \cite{Gro:87Hyp:aa} to study infinite groups as geometric objects.
 See \cite{Vai:05Gro:aa} for the basics of Gromov hyperbolic spaces for intrinsic metric spaces.
 In this paper, unless otherwise specified, we do not assume that a metric space is intrinsic or geodesic.

\subsection{Gromov Hyperbolic Spaces} \label{subsec:delta-hyp}

\begin{definition}\label{def:Gromov_prod}
Let $(X,d)$ be a metric space and let $w\in X$.  For $x,y\in X$, the Gromov product of $x$ and $y$ with respect to $w$ is defined to be \begin{equation*}(x\mid y)_w = \tfrac{1}{2}\left[d(x,w)+d(y,w)-d(x,y)\right].\end{equation*}
\end{definition}

\begin{definition}\label{defn:delta_hyp}
Let $\delta\ge 0$. The metric space $(X,d)$ is said to be {\it $\delta$-hyperbolic} if 
\begin{equation*}
(x\mid y)_w\ge\min\left\{(x\mid z)_w, (y\mid z)_w\right\}-\delta
\end{equation*}
 for all $x,y,z,w\in X$. A metric space $(X,d)$ is said to be {\it Gromov hyperbolic} if it is $\delta$-hyperbolic for some $\delta \geq 0$.                        
 \end{definition}

 An  inequality equivalent to that in Definition~\ref{defn:delta_hyp}, known as {\it the $4$-point inequality}, is given in the following proposition. 
 
\begin{proposition}[{\cite[Lemma 1.1.A]{Gro:87Hyp:aa}}]\label{prop:4points}
$(X,d)$ is $\delta$-hyperbolic if and only if  
\begin{equation*}\label{eq:4points}
d(x,y)+d(z,w)\le\max\{d(x,z)+d(y,w), d(y,z)+d(x,w)\}+2\delta
\end{equation*}
for all $x,y,z,w\in X.$ \qed
\end{proposition}

Two metric spaces $X$ and $Y$ are said to be {\it roughly similar} if there exists a (not necessarily continuous) map  $f\colon X\to Y$
and constants $\lambda>0$, $k\ge0$ such that $\sup_{y\in Y}d_Y(y,f(X))\le k$ and for all $x,x'\in X$ $$\lambda d_X(x,x')-k\le d_Y(f(x),f(x'))\le \lambda d_X(x,x')+k.$$  A straightforward argument shows that  Gromov hyperbolicity is preserved by rough similarity. 
\begin{proposition}\label{prop:rough-sim_hyp}
If  $X$ and $Y$ are roughly similar metric spaces then $X$ is Gromov hyperbolic if and only if $Y$ is Gromov hyperbolic.
 \end{proposition}
 \begin{proof}
Let $f\colon X\to Y$ be a $(\lambda,k)$-rough similarity, with $\lambda>0$ and $k\ge0$. Assume $X$ is \hbox{$\delta$-hyperbolic,} $\delta\ge0$. Since $f(X)$ is $k$-cobounded in $Y$, for any $x',y'\in Y$ there exist $x,y\in X$ such that $d_y(x',f(x))\le k$ and $d_Y(y',f(y))\le k$. Thus $d_Y(x',y')\le d_Y(x',f(x))+d_Y(f(x),f(y))+d_Y(y',f(y))\le d_Y(f(x),f(y))+2k\le \lambda d_X(x,y)+3k.$
For all $x',y',z',w'\in Y$ we have 
\begin{align*}
&d_Y(x',y')+d_Y(z',y') \le \lambda\left[d_X(x,y)+d_X(z,w)\right]+6k\\
&\le \lambda\left[\max\{d_X(x,z)+d_X(y,w),d_X(x,w)+d_X(y,z)\}+2\delta\right] +6k\\
&= \max\{\lambda d_X(x,z)+\lambda d_X(y,w),\lambda d_X(x,w)+\lambda d_X(y,z)\}+2\lambda\delta +6k\\
&\le \max\{d_Y(f(x),f(z))+d_Y(f(y),f(w)),d_Y(f(x),f(w))+d_Y(f(y),f(z))\}+2\lambda\delta +8k\\
&\le \max\{d_Y(x',z')+d_Y(y',w'),d_Y(x',w')+d_Y(y',z')\}+2\lambda\delta +12k.
\end{align*}
By Proposition~\ref{prop:4points}, $Y$ is $(\lambda\delta+6k)$-hyperbolic. The proof of the converse is similar. 
 \end{proof}

 Given constants $\lambda>0$ and $k\ge0$, we say that a function $\varphi\colon[0,\infty)\to[0,\infty)$ is a {\it $(\lambda,k)$-approximate dilation} if $|\varphi(t)-\lambda t|\le k$ for all $t\ge 0$.   
The function $\varphi$ is an {\it approximate dilation} if it is a \hbox{$(\lambda,k)$-approximate} dilation for some $\lambda>0$ and $k\ge0$.

\begin{remark}\label{rem:appr-dilation_rough-similar}
If $\varphi\in \mathcal{M}$ is a $(\lambda,k)$-approximate dilation then $(X,d)$ and $(X,d_\varphi)$ are $(\lambda,k)$-roughly similar. 
\end{remark}

Remark~\ref{rem:appr-dilation_rough-similar} and Proposition~\ref{prop:rough-sim_hyp} have the following consequence.

\begin{proposition}\label{prop:appr_dilation} 
If $(X,d)$ is a Gromov hyperbolic metric space and 
$\varphi\in\mathcal{M}$ is an approximate dilation then $(X,d_\varphi)$ is Gromov hyperbolic.\qed
\end{proposition}


\subsection{Approximately Ultrametric Spaces}\label{subsec:ultrametric}
 Recall that a metric space $(X, d)$ is  {\it ultrametric} if the metric $d$ satisfies the inequality:
 for all $x,y,z\in X$,  $d(x,y)\le\max\{d(x,z),d(y,z)\},$ 
a condition which implies the triangle inequality.
\begin{definition}\label{defn:delta_ultrametric}
Let $\delta\ge 0$. 
We say that a metric space $(X,d)$ is {\it $\delta$-ultrametric} if for all $x,y,z\in X$  
 \begin{equation*}
 d(x,y)\le \max\{d(x,z),d(y,z)\}+\delta.
 \end{equation*}
 We say that $(X,d)$ is {\it approximately ultrametric} if it is $\delta$-ultrametric for some $\delta \ge 0$. 
\end{definition}
Let $x,y,z\in X$ and let $s, m$ and $l$ denote the smallest, medium and largest of the distances $d(x,y)$, $d(y,z)$ and $d(x,z)$. Then the $\delta$-ultrametric condition is equivalent to 
$l-m\le \delta$.
Note that if $\delta=0$ this implies that $l=m$, exhibiting a well-known characteristic of ultrametric spaces, namely that triangles in such spaces are either acute isosceles (that is, the equal sides are the larger sides) or equilateral.
\linebreak 
If $\delta>0$ then any triangle triplet $(l_1,l_2,l_3)$ consisting of nonnegative numbers less or equal to $\delta$ satisfy the $\delta$-ultrametric condition, and if one of the numbers $l_i$ is greater than $\delta$ then there is at least one other number $l_j$, $j\neq i$ in the triplet satisfying $l_j\ge l_i-\delta$. In other words, in $\delta$-ultrametric spaces ``small triangles'' (with side length less than $\delta$) can have any shape, and ``large triangles'' (with one side length at least $\delta$) are acute $\delta$-almost isosceles or $\delta$-almost equilateral.  

The relationship between $\delta$-ultrametric and $\delta$-hyperbolic spaces is given by Proposition~\ref{prop:ultra-delta} below.
For this purpose, we need the following lemma.

\begin{lemma}\label{lem:sums}
Let $\delta\ge0$ and let $a_{ij}\in\bbR$, $i,j\in\{1,2,3,4\}$, be such that $a_{ij}=a_{ji}$.
\begin{itemize} 
\item[$(i)$]  If $a_{ij}\le\max\{a_{ik},a_{kj}\}+\delta$ for all $i,j,k$\, then \[a_{ij}+a_{kl}\le \max\{ a_{ik}+a_{jl}, a_{il}+a_{jk}\}+2\delta,\]
\item[$(ii)$] If $a_{ij}\ge\min\{a_{ik},a_{kj}\}-\delta$ for all  $i,j,k$\, then \[a_{ij}+a_{kl}\ge \min\{ a_{ik}+a_{jl}, a_{il}+a_{jk}\}-2\delta.\]
\end{itemize}
\end{lemma}

Note that if $L, M$ and $S$ denote the largest, medium and smallest of the sums $a_{ij}+a_{kl}$, $a_{ik}+a_{jl}$ and $a_{il}+a_{jk}$ for some choice of $i,j,k,l\in\{1,2,3,4\}$,  then the conclusion in part $(i)$ of the lemma is equivalent to $L-M\le2\delta$, and the one in part $(ii)$ to $M-S\le 2\delta$.

\begin{proof}
\noindent $(i)$ Fix $i,j,k,l\in\{1,2,3,4\}$. Without loss of generality, assume that $L=a_{ij}+a_{kl}$ is the largest sum and assume that $ a_{kl}\le a_{ij}$. Since  
$a_{ij}\le\max\{a_{ik},a_{kj}\}+\delta$ 
 and $a_{ij}\le\max\{a_{il},a_{lj}\}+\delta$, we have 
\begin{equation*}a_{ij}+a_{kl}\le 2a_{ij}\le\max\{a_{ik}+a_{il},a_{ik}+a_{lj},a_{kj}+a_{il},a_{kj}+a_{lj}\}+2\delta.\end{equation*}

\noindent If $a_{ik}\ge a_{kj}$ and $a_{lj}\ge a_{il}$ then 
$$M=a_{ik}+a_{lj} = \max\{a_{ik}+a_{il},a_{ik}+a_{lj},a_{kj}+a_{il},a_{kj}+a_{lj}\}$$
and if $a_{ik}\le a_{kj}$ and $a_{lj}\le a_{il}$ then  $$M=a_{kj}+a_{il} = \max\{a_{ik}+a_{il},a_{ik}+a_{lj},a_{kj}+a_{il},a_{kj}+a_{lj}\}.$$
In both cases, $L-M\le 2\delta$. Furthermore, if $a_{ik}\ge a_{kj}$ and $a_{lj}\le a_{il}$\, then $a_{ij}\le\max\{a_{ik},a_{kj}\}+\delta = a_{ik}+\delta$ and $a_{ij}\le\max\{a_{il},a_{lj}\}+\delta = a_{il}+\delta$, and since $a_{kl}\le\max\{a_{kj},a_{lj}\}+\delta$,
 \begin{align*}
 a_{ij}+a_{kl}& \le a_{ij}+\max\{a_{kj},a_{lj}\}+\delta = \max\{a_{ij}+a_{kj},a_{ij}+a_{lj}\}+\delta\\
 &\le \max\{a_{il}+\delta+a_{kj},a_{ik}+\delta+a_{lj}\}+\delta = \max\{a_{il}+a_{kj},a_{ik}+a_{lj}\}+2\delta.
 \end{align*}
\noindent Finally, if $a_{ik}\le a_{kj}$ and $a_{lj}\ge a_{il}$\, then $a_{ij}\le\max\{a_{ik},a_{kj}\}+\delta = a_{kj}+\delta$ and $a_{ij}\le\max\{a_{il},a_{lj}\}+\delta = a_{lj}+\delta$, and since $a_{kl}\le\max\{a_{ki},a_{il}\}+\delta$, we have $a_{ij}+a_{kl} \le a_{ij}+\max\{a_{ki},a_{il}\}+\delta\le\max\{a_{lj}+a_{ki},a_{kj}+a_{il}\}+2\delta$, that is, $L-M\le 2\delta.$

\noindent $(ii)$ Follows from $(i)$ by taking the negatives of $a_{ij}$. 
\end{proof}

\begin{proposition}\label{prop:ultra-delta} 
If $(X,d)$ is $\delta$-ultrametric then $(X,d)$ is $\delta$-hyperbolic.
\end{proposition}

\begin{proof}
Let $x_i, i=1,2,3,4,$ be four points in $X$. By part $(i)$ of Lemma~\ref{lem:sums}, with $a_{ij} = d(x_i,x_j)$,  \[d(x_i,x_j)+d(x_k,x_l)\le \max\{d(x_i,x_k)+d(x_j,x_l),d(x_i,x_l)+d(x_j,x_k)\}+2\delta\] and the conclusion follows from Proposition~\ref{prop:4points}.
\end{proof}

The case $\delta=0$ in Proposition~\ref{prop:ultra-delta}, that is, the fact that ultrametric spaces are $0$-hyperbolic, was  observed in \cite[(2.4)]{Ibr:12Mob:aa}.

\begin{remark}
The converse of Proposition~\ref{prop:ultra-delta} is not true.
For example, the Euclidean half line $([0,\infty), |\cdot|)$ is $0$-hyperbolic but not $\delta$-ultrametric for any $\delta\ge0$. To see this, let $x,y\ge 0$ and $z=(x+y)/2$.
Then the $\delta$-ultrametric condition  is equivalent to $|x-y|\le 2\delta$, which cannot be valid for all $x,y\ge0$.
\end{remark}

\begin{definition}
Let $\eta\ge0$. We say that a function $\varphi\colon [0,\infty)\to \IR$ is {\it $\eta$-nondecreasing} if $t\le s$ implies $\varphi(t)\le \varphi(s)+\eta$.
\end{definition}

Observe that if $\varphi\colon [0,\infty) \to \IR$ is $\eta$-nondecreasing then the function $\varphi^+$ given by $\varphi^+(t) = \sup\{ \varphi(s)  \mid s \leq t\}$ is nondecreasing and satisfies
$0 \leq \varphi^+(t) - \varphi(t) \leq \eta$.  

We say that the function $\varphi\colon [0,\infty) \to\IR$ is {\it approximately nondecreasing} if $\varphi$ is $\eta$-nondecreasing for some $\eta\ge0$. 
 Note that $\varphi$ is approximately nondecreasing if and only if $\varphi$ is within a bounded distance from a nondecreasing function.

\begin{proposition}\label{prop:appr_log-like}
Let $(X,d)$ be a metric space and let $\delta,\eta\ge 0$. If $\varphi\in \mathcal{M}$ is $\eta$-nondecreasing and satisfies $\varphi(2t)-\varphi(t)\le \delta$ for all $t\ge 0$\, then $(X,d_\varphi)$ is $(\delta+2\eta)$-ultrametric.
\end{proposition}

\begin{proof} For any $x,y,z\in X$, \begin{align*}
d_\varphi(x,y)& = \varphi(d(x,y))\\
& \le\varphi(d(x,z)+d(y,z))+\eta \quad \text{since $\varphi$ is $\eta$-nondecreasing}\\
& \le\varphi(\max\{2d(x,z),\, 2d(y,z)\})+2\eta  \\ 
& \le\max\{\varphi(2d(x,z)),\, \varphi(2d(y,z))\} +2\eta  \\
&\le \max\{\varphi(d(x,z))+\delta,\, \varphi(d(y,z))+\delta\} +2\eta \quad \text{since } \varphi(2t)-\varphi(t)\le \delta\\
& = \max\{d_\varphi(x,z),\, d_\varphi(y,z)\}+\delta+2\eta
\end{align*}
which shows that $(X,d_\varphi)$ is $(\delta+2\eta)$-ultrametric.
\end{proof}

\begin{corollary}[{\cite[Example 1.2(c)]{Gro:87Hyp:aa}}]\label{cor:ln-ultra}
Let  $\varphi(t)=\log(1+t)$, $t\ge 0$.
For any metric space $(X,d)$, the transformed metric space $(X, d_\varphi)$ is $\log(2)$-ultrametric and so by Proposition~\ref{prop:ultra-delta} is $\log(2)$-hyperbolic.
\end{corollary}

\begin{proof}
 $ \varphi(2t)-\varphi(t) = \log(1+2t)-\log(1+t) = \log\left(\frac{1+2t}{1+t}\right)<\log(2)$ for all $t\ge0$.
This inequality is sharp since $\lim_{t\to\infty}\log\left(\frac{1+2t}{1+t}\right) = \log(2).$ 
\end{proof}

Let $k\ge 0$. Recall that a {\it $k$-rough geodesic} in a metric space $(X,d)$ is a $k$-rough isometric embedding of an interval $I\subseteq \IR$ into $X$.
That is, a map $\gamma\colon I \to X$ (not necessarily continuous) such that for all $t,s\in I$, $$|t-s|-k\le d(\gamma(t),\gamma(s))\le |t-s|+k.$$ 
The space $X$ is called {\it $k$-roughly geodesic} if for every $x,y\in X$ there exists a $k$-roughly geodesic segment joining $x$ and $y$; and $X$ is called {\it roughly geodesic} if it is $k$-roughly geodesic for some $k\ge 0$.
Furthermore, we say that a metric space $(X,d)$ has the {\it $k$-rough midpoint property} if for every $x,y\in X$ there exists $z\in X$ such that $$\max\{d(x,z),d(y,z)\}\le \tfrac{1}{2}d(x,y)+k.$$
 A space has the {\it rough midpoint property} if it has the $k$-rough midpoint property for some $k\ge 0$. 
The following lemma asserts that the rough midpoint property is a necessary condition for a space to be roughly geodesic.

\begin{lemma}\label{lem:rough_geod_midp}
If $(X,d)$ is roughly geodesic then it has the rough midpoint property.
\end{lemma}
\begin{proof}
Assume $(X,d)$ is $k$-roughly geodesic for some $k\ge 0$. Let $x,y\in X$.
Let $\gamma\colon [0,b]\to X$ be a $k$-rough geodesic segment joining $x$ and $y$ where $b\ge 0$.
Note that $|b- d(x,y) |  \le k$.
Let $z=\gamma(\tfrac{1}{2}b)$. 
Then
\[
 d(x,z), \, d(y,z) ~\le~ \tfrac{1}{2}b +k ~\le~ \tfrac{1}{2}d(x,y) +  \tfrac{3}{2}k.
\]
Hence $X$ has the $\tfrac{3}{2}k$-rough midpoint property.
\end{proof}

Our next result asserts that an unbounded, approximately ultrametric space  cannot be roughly geodesic.

\begin{proposition}\label{prop:appr_ultra_non_geo}
If   $(X,d)$ is unbounded  and approximately ultrametric  then $X$ is \hbox{not roughly geodesic.}
\end{proposition}
\begin{proof}
Suppose $(X,d)$ is unbounded, $\delta$-ultrametric and $k$-roughly geodesic  for some given $\delta,k\ge0$.
By Lemma~\ref{lem:rough_geod_midp} and its proof, $X$ has the $\tfrac{3}{2}k$-rough midpoint property.
Thus, for any $x,y\in X$, there exists $z\in X$ such that $\max\{d(x,z),d(y,z)\}\le \tfrac{1}{2}d(x,y)+ \tfrac{3}{2}k$.
The $\delta$-ultrametric inequality (\ref{defn:delta_ultrametric}) implies that $$d(x,y)\le \max\{d(x,z),d(y,z)\}+\delta\le \tfrac{1}{2}d(x,y)+\tfrac{3}{2}k+\delta,$$
hence $d(x,y)\le 3k+2\delta$, contradicting the hypothesis that $X$ is unbounded.
\end{proof}


\section{Concave Metric Transforms of the Euclidean Half Line}\label{sec:conc_mt}

Let $\mathcal{C}$ denote the class of unbounded concave functions $\varphi\colon [0,\infty)\to[0,\infty)$ satisfying $ \lim_{t\to0^+}\varphi(t)=\varphi(0)=0$. 
Note that
if $\varphi\in\mathcal{C}$ then $\varphi$ is strictly increasing, continuous on $[0,\infty)$, and differentiable on $(0,\infty)$ except possibly at a countable number of points. 
In this section we give a simple characterization of all functions $\varphi\in \mathcal{C}$ for which the transformed Euclidean half line $([0,\infty),|\cdot|_\varphi)$ is Gromov hyperbolic (Theorem~\ref{thm:conc_mainA}).  

For $\varphi\in \mathcal{C}$, the Gromov product based at $0$ for the transformed Euclidean metric $|x-y|_\varphi = \varphi(|x-y|)$ on the half line $[0,\infty)$ is given by
\begin{equation} \label{eq:phi_prod}
(x\mid_\varphi y)=(x\mid_\varphi y)_0 = \tfrac{1}{2}[\varphi(x)+\varphi(y)-\varphi(|x-y|)].
\end{equation}
Let $\delta\ge 0$ and assume that  $(X,d_\varphi)$ is $\delta$-hyperbolic. Then
\begin{equation}\label{eq:phi_hyp}
(x\mid_\varphi y)\ge \min\{(x\mid_\varphi z),(y\mid_\varphi z)\}-\delta, \mbox{ for all } x,y,z\ge 0.
\end{equation}

We  investigate the restrictions on $\varphi$ imposed by the inequality (\ref{eq:phi_hyp}).

\begin{lemma}\label{lem:prod_monotonic} Let $\varphi\in \mathcal{C}$ and fix $a\ge 0$. The function \[ x\mapsto (a\mid_\varphi x) = \tfrac{1}{2}[\varphi(a)+\varphi(x)-\varphi(|a-x|)]\]  is strictly increasing on $[0,a]$, and decreasing on $[a,\infty)$. \end{lemma}

\begin{proof}  
If $0\le x_1< x_2\le a$\, then
\[
2(a\mid_\varphi x_2)-2(a\mid_\varphi x_1)  = \varphi(x_2)-\varphi(x_1) + \varphi(a-x_1)-\varphi(a-x_2) > 0\] since $\varphi$ is strictly increasing. Thus $x\mapsto (a\mid_\varphi x)$ is strictly increasing on $[0,a]$. 
If $0\le a\le x_1< x_2$\, then
\[2(a\mid_\varphi x_2)-2(a\mid_\varphi x_1) = (x_2-x_1)\left[\frac{\varphi(x_2)-\varphi(x_1)}{x_2-x_1} - \frac{\varphi(x_2-a)-\varphi(x_1-a)}{(x_2-a)-(x_1-a)}\right]\le 0\]  since the quantity in the square brackets is nonpositive by  concavity. Thus $x\mapsto (a\mid_\varphi x)$ is decreasing on $[a,\infty)$.
\end{proof}

By Lemma~\ref{lem:prod_monotonic},
for given $x,y\ge 0$, the minimum in the right side of  (\ref{eq:phi_hyp}) is attained at $\max\{x,y\}$ when $z\le x,y\,$ and at $\min\{x,y\}$ when $x,y\le z$. For the case when $x\le z\le y$, or $y\le z\le x$, we consider the equation $(x\mid_\varphi z)=(y\mid_\varphi z)$.
The solution of this equation is the objective of our next lemma.

 \begin{lemma}\label{lem:w}
Let $\varphi\in \mathcal{C}$. For each $0\le x< y$, there exists a unique $\omega=\omega{(x,y)}$ with $ x\le \omega\le\min\left\{\frac{x+y}{2},2x\right\}$ such that  $ (x\mid_\varphi \omega)=(y\mid_\varphi \omega).$
Moreover, $\omega(x,y)=x$  for all $0\le x\le y$
if and only if $\varphi$ is a dilation, that is,  if
$\varphi(x)=\lambda x$ for some $\lambda >0$.
 \end{lemma} 
\begin{proof}  For $z\in [x,y]$, the equation $ (x\mid_\varphi z)=(y\mid_\varphi z)$ rewrites as
 \begin{equation}\label{eq:w}
\varphi(x)-\varphi(y)+\varphi(y-z)-\varphi(z-x)=0.
\end{equation}
 Let 
$f(z)  = 2[ (x\mid_\varphi z)-(y\mid_\varphi z)] = \varphi(x)-\varphi(y)+\varphi(y-z)-\varphi(z-x).$
\noindent By Lemma~\ref{lem:prod_monotonic}, $z\mapsto (x\mid_\varphi z)$ is decreasing and $z\mapsto (y\mid_\varphi z)$ is strictly increasing, and so the function $f(z)$ is strictly decreasing on $[x,y]$. Furthermore, \[f(x) = \varphi(x)-\varphi(y)+\varphi(y-x)\ge 0\] since $\varphi$ is subadditive, and \[f(y) = \varphi(x) - \varphi(y)-\varphi(y-x)< 0\] since $\varphi$ is strictly increasing. The function $f(z)$ is continuous and one-to-one on the interval $[x,y]$ and, by the Intermediate Value Theorem, there exists a unique $w\in [x,y]$ such that $f(\omega)=0$.
Observe that  \[f\left(\tfrac{x+y}{2}\right) = \varphi(x)-\varphi(y)\le 0\] and, if $y\ge 2x$ then \[f(2x) = \varphi(y-2x)-\varphi(y)\le 0,\] hence $ x\le \omega\le\min\left\{\tfrac{x+y}{2},2x\right\}$. 
In order to prove the last part of the lemma, assume that $\omega=x$ satisfies (\ref{eq:w}). Then $\varphi(y)=\varphi(x) +\varphi(y-x),$ for all $0\le x\le y$, which shows that $\varphi$ is additive.
Consequently,  $\varphi(rx)=r\varphi(x)$ for any nonnegative rational number $r$ and all $x\ge0$.
Since $\varphi$ is continuous, it follows that $\varphi(tx)=t\varphi(x)$                          
for all $t,x\ge0$, which shows that $\varphi$ is also homogenous. 
Thus $\varphi$ is linear and, since $\varphi$ is unbounded, $\varphi(x)=\lambda x$ for some $\lambda >0$. 
The converse is evident.
\end{proof}

\begin{lemma}\label{lem:w_C1}
Assume $\varphi\in\mathcal{C}$ is not a dilation. If $\varphi$ is of class $C^1$ on $(0,\infty)$ then the solution $\omega=\omega(x,y)$ given by Lemma~\ref{lem:w} is increasing as function of $y$.
\end{lemma}

\begin{proof}
 Let $ F(x,y,z)=\varphi (x)-\varphi (y)+\varphi (y-z)-\varphi (z-x)$. Then $F$ is of class $C^1$ on the open subset $\{(x,y,z)\in \IR^3\mid 0<x<z<y\}$ and by Lemma~\ref{lem:w} there exists $x<\omega<y$ such that $ F(x,y,\omega)=0$. Furthermore, since $x< \omega < \min\{\frac{x+y}{2},2x\}$ and $\varphi' $ is strictly decreasing, we have 
\begin{align*}
\partial_x F|_{(x,y,\omega)}   &= \varphi' (x)+\varphi' (\omega -x)>0,\\
\partial_y F|_{(x,y,\omega)}   &= -\varphi' (y)+\varphi' (y-\omega )>0,\\
\partial_z F|_{(x,y,\omega)}   &=  -\varphi' (y-\omega )   -  \varphi' (\omega -x) <0. 
\end{align*}
Then, by the Implicit Function Theorem the solution $\omega  = \omega (x,y)$ is of class $C^1$. Taking the derivative with respect to $y$ in $ F(x,y,\omega ) = 0$ gives \[-\varphi' (y)+\varphi' (y-\omega )(1-\partial_y\omega )-\varphi' (\omega -x)\partial_y\omega  = 0.\] Thus \[\partial_y \omega  = \frac{\varphi' (y-\omega )-\varphi' (y)}{\varphi' (y-\omega )+\varphi' (\omega -x)}> 0\] which shows that $\omega  = \omega (x,y)$ is increasing as a function of $y$ for all $0<x<y$. 
\end{proof}

\begin{proposition}\label{prop:w_hat}
Let $\varphi\in\mathcal{C}$ and let $\lambda =\lim_{t\to\infty}\varphi'_-(t)$. For each $x\ge0$ there exists a unique $\widehat{\omega}=\widehat{\omega}{(x)}$ with $x\le \widehat{\omega}\le 2x$ such that 
\begin{equation}\label{eq:w_hat}
\varphi(x)-\varphi(\widehat{\omega}-x)=\lambda \widehat{\omega}.
\end{equation}
\end{proposition}

\begin{proof} Note that if $\varphi(x)=\lambda x$ then by Lemma~\ref{lem:w}, $\widehat{\omega} = \omega = x$. For the remainder of the proof we assume that $\varphi$ is not a dilation. Fix $x\ge 0$. For each $y>x$, let $\omega=\omega(x,y)$ be the solution of \[f(\omega)=\varphi(x)-\varphi(y)+\varphi(y-\omega)-\varphi(\omega-x)=0\]
given by Lemma~\ref{lem:w}. We show that $ \widehat{\omega}=\widehat{\omega}(x) = \lim_{y\to\infty}\omega(x,y)$.
To prove that this limit exists, we use of the uniform approximation of $\varphi$ given by Corollary~\ref{cor:conc_approx}.
For this, let $\varepsilon>0$ and let $\psi_\varepsilon\colon [0,\infty)\to[0,\infty)$ be a continuous concave function, which is of class $C^1$ on $(0,\infty)$ and satisfies $\psi_\varepsilon(0)=0$ and $|\varphi- \psi_\varepsilon|\le\varepsilon$.  Note that $\lim_{t\to\infty}\psi_\varepsilon'(t) = \lim_{t\to\infty}\varphi_-'(t)=\lambda$.
We can assume that $\psi_\varepsilon$ is not linear (not a dilation), for otherwise if $\psi_\varepsilon$ were linear for  arbitrarily small $\varepsilon$ then $\varphi$ would be a dilation.
\linebreak 
By Lemma~\ref{lem:w_C1}, there exists a unique $\upsilon_\varepsilon = \upsilon_\varepsilon(x,y)$ such that $x< \upsilon_\varepsilon< \min\{\frac{x+y}{2},2x\}$ and satisfying \[g_\varepsilon(\upsilon_\varepsilon) = \psi_\varepsilon(x)-\psi_\varepsilon(y)+\psi_\varepsilon(y-\upsilon_\varepsilon)-\psi_\varepsilon(\upsilon_\varepsilon-x)=0.\]
 Let $\widehat{\upsilon}_\varepsilon =\widehat{\upsilon}_\varepsilon(x) = \lim_{y\to\infty}\upsilon_\varepsilon(x,y).$
This limit exists because $\upsilon_\varepsilon$ is increasing as a function of $y$ and it is bounded from above by $2x$ as $y\to\infty$.
Taking the limit as $y\to\infty$ in the expression $g_\varepsilon(\upsilon_\varepsilon) =0$ yields
\begin{align*}
0 & = \lim_{y\to\infty} [\psi_\varepsilon (x)-\psi_\varepsilon (y)+\psi_\varepsilon (y-\upsilon_\varepsilon )-\psi_\varepsilon (\upsilon_\varepsilon -x)]\\
& = \lim_{y\to\infty}[\psi_\varepsilon (x)-\psi_\varepsilon (\upsilon_\varepsilon -x) -\upsilon_\varepsilon  \psi_\varepsilon' (\eta_\varepsilon)] \\ 
& = \psi_\varepsilon (x)-\psi_\varepsilon (\widehat{\upsilon}_\varepsilon-x) - \lambda\,\widehat{\upsilon}_\varepsilon
\end{align*}
where $y-\upsilon_\varepsilon  <\eta_\varepsilon<y$ is given by the Mean Value Theorem, and $\eta_\varepsilon\to\infty$ as $y\to \infty$. 
Since $|\varphi- \psi_\varepsilon|\le\varepsilon$, we have that $| g_\varepsilon-f|\le 4\varepsilon$ and in particular, $|f(\omega)-f(\upsilon_\varepsilon)|=| f(\upsilon_\varepsilon)|\le 4\varepsilon$. 
Taking the limit as $\varepsilon\to0$ and using the fact that $f$ is one-to-one on $[x,y]$, it follows that $\omega(x,y)=\lim_{\varepsilon\to0}\upsilon_\varepsilon(x,y)$.  
Taking the limit as $y\to\infty$ gives \[\widehat{\omega} = \lim_{\varepsilon\to0}\widehat{\upsilon}_\varepsilon = \lim_{y\to\infty}\omega(x,y)\] satisfying
$\varphi(x)-\varphi(\widehat{\omega}-x)-\lambda \widehat{\omega} =0.$ \end{proof}

The above observations allow us to show the following.

\begin{proposition}\label{prop:phi_half_line}
Let $\delta\ge0$ and let $\varphi\in\mathcal{C}$ be such that $([0,\infty),|\cdot|_\varphi)$ is $\delta$-hyperbolic. 
Then $\varphi$ satisfies 
\begin{equation}\label{eq:phi_hyp_cond}
\varphi(\widehat{w})-\varphi(\widehat{w}-x)\le \lambda x + 2\delta
\end{equation} for all $x\ge 0$, where $\lambda =\lim_{t\to\infty}\varphi'_-(t)$ and $\widehat{w}=\widehat{w}{(x)}$, $x\le \widehat{w}\le 2x$ is the unique solution of 
\begin{equation}\label{eq:w_hat_implicit}
\varphi(x)-\varphi(\widehat{w}-x)=\lambda \widehat{w}.
\end{equation}
\end{proposition}

\begin{proof}  Let $\delta$ and $\varphi$ be as in the statement of the proposition. Recall that if the space $([0,\infty),|\cdot|_\varphi)$ is $\delta$-hyperbolic then $\varphi$ satisfies the inequality (\ref{eq:phi_hyp}) \[(x\mid_\varphi y)\ge \min\{(x\mid_\varphi z),(y\mid_\varphi z)\}-\delta\] for all $x,y,z\ge 0.$  We show that this condition implies (\ref{eq:phi_hyp_cond}). 
Without loss of generality, we assume that $0\le x\le y$. Then there are three possible cases for $z.$
\vskip.2cm
\noindent {\it Case 1}. Assume $z\le x\le y.$ By Lemma~\ref{lem:prod_monotonic}, $(z\mid_\varphi y)\le (z\mid_\varphi x)$ and $(x\mid_\varphi y) \ge (z\mid_\varphi y)$. 
Hence,   $(x\mid_\varphi y)\ge \min\{(x\mid_\varphi z),(y\mid_\varphi z)\}$,  that is, the condition (\ref{eq:phi_hyp}) holds with $\delta=0$ and for all $\varphi$.
\vskip.2cm
\noindent {\it Case 2}. Assume $x\le y\le z.$ By Lemma~\ref{lem:prod_monotonic}, $(x\mid_\varphi z)\le (x\mid_\varphi y)$ and $(y\mid_\varphi z)\le (x\mid_\varphi y)$, which implies that $(x\mid_\varphi y)\ge \min\{(x\mid_\varphi z),(y\mid_\varphi z)\}.$ As before, the condition (\ref{eq:phi_hyp}) holds with $\delta=0$ and for all $\varphi$.
\vskip.2cm

\noindent {\it Case 3}. Assume $x\le z\le y$. Let $\omega=\omega{(x,y)}$ be the unique value $x\le \omega\le \min\left\{\frac{x+y}{2},2x\right\}$ satisfying (\ref{eq:w}) as given by Lemma~\ref{lem:w}.

Consider the following two possible situations. 

\noindent {\it Case 3(a)}. Assume $x\le z\le \omega\le y$. Then $\min\{(x\mid_\varphi z),(y\mid_\varphi z)\} = (y\mid_\varphi z)$ and the inequality (\ref{eq:phi_hyp}) becomes $(x\mid_\varphi y)\ge (y\mid_\varphi z)-\delta,$ or equivalently  
\begin{equation*}
\varphi(z)-\varphi(x) +\varphi(y-x)-\varphi(y-z)\le 2\delta.
\end{equation*}  
For $z\in [x,\omega]$, let $g(z) = 2[(y\mid_\varphi z) - (y\mid_\varphi x)] = \varphi(z)-\varphi(x)+\varphi(y-x)-\varphi(y-z).$ From Lemma~\ref{lem:prod_monotonic}, the function $g(z)$ is increasing on $[x,\omega]$, and hence $\max_{z\in [x,\omega]} g(z) = g(\omega).$ Thus it suffices to find conditions on $\varphi$ such that 
\begin{equation*}\label{eq:phi_c3.1}
 g(\omega) = 
 \varphi(\omega)-\varphi(x)+\varphi(y-x)-\varphi(y-\omega)\le 2\delta 
\end{equation*} 
for all $0\le x\le y$. Taking the limit as $y\to\infty$ in the above inequality and letting  $\lambda =\lim_{t\to\infty}\varphi'_-(t)$ yields
\begin{equation*}\label{eq:w_hat_delta}
\varphi(\widehat{\omega})-\varphi(x)+\lambda (\widehat{\omega}-x)\le 2\delta
\end{equation*} where $\widehat\omega=\widehat{\omega}(x)$ is given by Proposition~\ref{prop:w_hat}.
Combining with (\ref{eq:w_hat}), this gives 
\begin{equation*}\label{eq:w_hat_hyp}
\varphi(\widehat{\omega})-\varphi(\widehat{\omega}-x)\le \lambda x+2\delta.
\end{equation*}

\noindent {\it Case 3(b)}. Assume $x\le \omega\le z\le y$. Then $\min\{(x\mid_\varphi z),(y\mid_\varphi z)\} = (x\mid_\varphi z)$ and the inequality (\ref{eq:phi_hyp}) becomes $(x\mid_\varphi y)\ge (x\mid_\varphi z)-\delta,$ 
or equivalently  
\begin{equation*}
\varphi(z)-\varphi(y) +\varphi(y-x)-\varphi(z-x)\le 2\delta.
\end{equation*}  
For $z\in [w,y]$, let  $h(z) = 2[(x\mid_\varphi z) - (y\mid_\varphi x)] = \varphi(z)-\varphi(y)+\varphi(y-x)-\varphi(z-x).$ By Lemma~\ref{lem:prod_monotonic}, the function $h(z)$ is decreasing on $[\omega,y]$ and since $\max_{z\in [\omega,y]} h(z) = h(\omega)$ it suffices to find conditions on $\varphi$ such that 
\begin{equation*}\label{eq:phi_c3.2}
 h(\omega) =  
 \varphi(\omega)-\varphi(y)+\varphi(y-x)-\varphi(\omega-x)\le 2\delta 
\end{equation*} 
for all $0\le x\le y$. Taking the limit as $y\to\infty$  yields 
\begin{equation*}
\varphi(\widehat{\omega})-\varphi(\widehat{\omega}-x)-\lambda x\le 2\delta
\end{equation*} where, as before, $\widehat\omega=\widehat{\omega}(x)$ 
is given by Proposition~\ref{prop:w_hat}.
\end{proof}

As noted in Remark~\ref{rem:phi'_lim}, if $\lambda =\lim_{x\to\infty}\varphi'_-(x)$ then $\lambda \ge0$, and as we will next see the cases $\lambda =0$ and $\lambda >0$ define mutually disjoint classes of functions.

Consider first the case $\lambda =0$. Then, from (\ref{eq:w_hat_implicit}), $\varphi(x)=\varphi(\widehat{\omega}-x)$ and since $\varphi$ is one-to-one, this implies that $\widehat{\omega}=2x$. In this case, the condition (\ref{eq:phi_hyp_cond}) becomes $\varphi(2x)-\varphi(x)\le 2\delta,$ and we have the following.

\begin{corollary}\label{cor:phi_hyp_0}
Let $\delta\ge0$.
Let $\varphi\in\mathcal{C}$ be such that $ \lim_{x\to\infty}\varphi'_-(x)=0$ and $([0,\infty),|\cdot|_\varphi)$ is $\delta$-hyperbolic.  Then for all $x\ge 0$,
\begin{equation}\label{eq:phi_hyp_0}
\varphi(2x)-\varphi(x)\le 2\delta.
\end{equation}
\end{corollary}

Condition (\ref{eq:phi_hyp_0}) is equivalent to $\varphi'_-(x)\le M/x$ for all $x>0$, where $M\ge 0$ is a constant depending on $\delta$.
We say that a function  satisfying  (\ref{eq:phi_hyp_0}) is {\it logarithm-like}.
The constants $\delta$ and $M$ are related as follows.

\begin{proposition}[Characterization of logarithm-like concave functions]  
Let $\varphi\in\mathcal{C}$.  Then
\begin{itemize}
\item[$(i)$] If $\varphi(2x)-\varphi(x)\le 2\delta$ for all $x\ge0$\, then $\varphi'_-(x)\le  4\delta/x$ for all $x>0$,
\item[$(ii)$] If $\varphi'_-(x)\le  M/x$ for all $x>0$\, then $\varphi(2x)-\varphi(x)\le M\log(2)$ for all $x\ge0$.
\end{itemize}
\end{proposition}

\begin{proof}
\noindent$(i)$ If $x>0$\, then by Proposition~\ref{prop:conc_diff},  \[\varphi'_-(2x)\le\frac{\varphi(2x)-\varphi(x)}{2x-x}\le\frac{2\delta}{x}\]
 or equivalently, $\varphi'_-(x)\le 4\delta/x, \mbox{ for all } x>0.$\\
\noindent$(ii)$ If $\varphi'_-(t)\le M/t$ for all $t>0$\, then, by Corollary~\ref{cor:conc_FT}, integration over $[x,2x]$ with $x>0$ yields $\varphi(2x)-\varphi(x)\le M\log(2)$.
\end{proof}

We now consider the case $\lambda > 0$. In this case (\ref{eq:phi_hyp_cond}) together with (\ref{eq:w_hat_implicit}) implies that \[\varphi(\widehat{\omega})-\varphi(x)\le \lambda (x-\widehat{\omega})+2\delta\] and since $\varphi$ is increasing and $x\le\widehat{\omega}$, this yields $\lambda (x-\widehat{\omega})+2\delta\ge 0$, or equivalently $0\le \widehat{\omega}-x\le 2\delta/\lambda .$
Together with (\ref{eq:phi_hyp_cond}) this gives \[\varphi(\widehat{\omega})\le \lambda x +\varphi(\widehat{\omega}-x)+2\delta \le \lambda x + \varphi(2\delta/\lambda )+2\delta\] which implies \[\varphi(x)\le \lambda x+\varphi(2\delta/\lambda )+2\delta.\]
Furthermore, since $\varphi$ is concave, the condition $\lambda >0$ implies that $\lambda x\le \varphi(x)$ for all $x\ge0$. Thus we have the following.

\begin{corollary}\label{cor:phi_hyp_>0}
Let $\delta\ge0$. Let  $\varphi\in\mathcal{C}$ be such that $ \lambda =\lim_{x\to\infty}\varphi'_-(x)>0$ and  $([0,\infty),|\cdot|_\varphi)$ is $\delta$-hyperbolic. Then  for all $x\ge 0$ 
\begin{equation}\label{eq:phi_hyp_>0}
\lambda x\le\varphi(x)\le \lambda x+\varphi(2\delta/\lambda )+2\delta.
\end{equation}
\end{corollary}

\begin{remark}\label{rem:psi}
If a function $\varphi\in \mathcal{C}$ satisfies the conditions of the preceding corollary then $\varphi(x) = \lambda x + f(x)$, where $f\colon[0,\infty)\to[0,\infty)$ is a bounded continuous concave function satisfying $f(0)=0$ and $0\le f(x)\le \varphi(2\delta/\lambda )+2\delta$ for all $x\ge0$. In particular, $\varphi$ is a  $(\lambda ,k)$-approximate dilation with $k=\varphi(2\delta/\lambda )+2\delta$.
\end{remark}

Consequently, we obtain the following characterization of unbounded continuous concave functions $\varphi\colon[0,\infty)\to[0,\infty)$ satisfying $\varphi(0)=0$ for which the transformed Euclidean metric $|x-y|_\varphi = \varphi(|x-y|)$ on $[0,\infty)$ is Gromov hyperbolic.

\begin{theorem}\label{thm:conc_mainA}
Let  $\varphi\in\mathcal{C}$ and let $\lambda =\lim_{x\to\infty}\varphi'_-(x)$. The transformed Euclidean half line \hbox{$([0,\infty),|\cdot|_\varphi)$} is Gromov hyperbolic if and only if $\varphi$ has of one of the following forms:
\begin{itemize}
\item[$(i)$] $\lambda >0$ and $\varphi(x) = \lambda x+f(x)$, where $f$  is a  nonnegative, bounded, continuous concave function satisfying $f(0)=0$, or

\item[$(ii)$] $\lambda =0$ and $\varphi(2x)-\varphi(x)$ is bounded.
\end{itemize}
\end{theorem}
\begin{proof}
Let $\varphi\in\mathcal{C}$ and let $\lambda =\lim_{x\to\infty}\varphi'_-(x)$. If $([0,\infty),|\cdot|_\varphi)$ is $\delta$-hyperbolic for some $\delta\ge0$, then the conclusion follows from Corollary~\ref{cor:phi_hyp_>0} and Remark~\ref{rem:psi} if $\lambda >0,$ and from Corollary~\ref{cor:phi_hyp_0} if $\lambda =0$. Conversely, if $\varphi$ has form $(i)$ then $\varphi$ is an approximate dilation and since the Euclidean half line is $0$-hyperbolic, the transformed space $([0,\infty),|\cdot|_\varphi)$ is Gromov hyperbolic by Proposition~\ref{prop:appr_dilation}. If $\varphi$ is of form $(ii)$ then $\varphi$ is logarithm-like and $([0,\infty),|\cdot|_\varphi)$ is approximately ultrametric by Proposition~\ref{prop:appr_log-like} and therefore Gromov hyperbolic by Proposition~\ref{prop:ultra-delta}.
\end{proof}


\section{Approximately Nondecreasing Metric Transforms of the Euclidean Half Line}\label{sec:app_conc_mt}

In this section we extend Theorem~\ref{thm:conc_mainA} to the more general class of approximately nondecreasing metric transforms (Theorem~\ref{thm:appr_nondecr_mainA}).

Recall $\mathcal{M}$ is  the class of all metric transforms.  
Observe that since for all $0\le s\le t$ the triplets $(\frac{t+s}{2},\frac{t-s}{2},s)$ and $(\frac{t+s}{2},\frac{t-s}{2},t)$ are triangle triplets, any $\varphi\in \mathcal{M}$ satisfies the inequality
$$\left|\varphi\left(\tfrac{t+s}{2}\right)-\varphi\left(\tfrac{t-s}{2}\right)\right|\le\tfrac{1}{2}\varphi(t)+\tfrac{1}{2}\varphi(s)\le \varphi\left(\tfrac{t+s}{2}\right)+\varphi\left(\tfrac{t-s}{2}\right).$$

Our next proposition shows that the requirement that the transformed Euclidean half line \hbox{$([0,\infty),|\cdot|_\varphi)$} is Gromov hyperbolic imposes additional conditions on the metric transform $\varphi$.

\begin{proposition}\label{prop:delta-mt_ineq}
Let $\delta\ge 0$. If $\varphi\in\mathcal{M}$ is such that $([0,\infty),|\cdot|_\varphi)$ is $\delta$-hyperbolic
then for all $0\le s\le t$, $$\left|\varphi\left(\tfrac{t+s}{2}\right)-\varphi\left(\tfrac{t-s}{2}\right)\right|\le\tfrac{1}{2}\varphi(t)+\tfrac{1}{2}\varphi(s)\le \max\left\{\varphi\left(\tfrac{t+s}{2}\right),\varphi\left(\tfrac{t-s}{2}\right)\right\}+\delta.$$
\end{proposition}

\begin{proof}
Let $\delta\ge0 $ and let $\varphi\in \mathcal{M}$ be such that  $([0,\infty),|\cdot|_\varphi)$ is $\delta$-hyperbolic. The four point condition for the transformed metric $|\cdot|_\varphi$ (see Proposition~\ref{prop:4points}), implies that $\varphi$ satisfies the following inequality  $$\varphi(|x-y|)+\varphi(|z-w|)\le \max\{\varphi(|x-z|)+\varphi(|y-w|),\varphi(|x-w|)+\varphi(|y-z|)\}+2\delta$$ for all $x,y,z,w\ge 0$.
Taking $w=0$ and $z=x+y$ yields  $$\varphi(|x-y|)+\varphi(x+y)\le \max\{2\varphi(y),2\varphi(x)\}+2\delta,$$ and by letting $t=x+y$ and $s=|x-y|$, we have that $0\le s\le t$ and
\begin{equation*}
\varphi(s)+\varphi(t)\le2\max\left\{\varphi\left(\tfrac{t+s}{2}\right),\varphi\left(\tfrac{t-s}{2}\right)\right\}+2\delta.\qedhere
\end{equation*}
\end{proof}

The following proposition shows that approximately nondecreasing metric transforms $\varphi\in\mathcal{M}$ for which $([0,\infty),|\cdot|_\varphi)$ is Gromov hyperbolic are approximately midpoint-concave and therefore, by Corollary~\ref{cor:appr-midconc_conc}, approximately concave. 

Recall that a function $\varphi\colon [0,\infty)\to \IR$ is approximately nondecreasing if there exists $\eta\ge0$ such that  $\varphi(t)\le \varphi(s)+\eta$ whenever $0\le t\le s$. 
The function $\varphi$ is approximately midpoint-concave if there exists $\delta\ge0$ such that $\frac{1}{2}\varphi(t)+\frac{1}{2}\varphi(s)\le \varphi\left(\frac{t+s}{2}\right)+\delta$ for all $t,s\ge 0$.

\begin{proposition}\label{prop:incr_hyp_mt}
Assume $\varphi\in\mathcal{M}$ is a $\eta$-nondecreasing 
metric transform such that $([0,\infty),|\cdot|_\varphi)$ is $\delta$-hyperbolic. Then there exists a continuous concave metric transform $\psi\in \mathcal{C}$ such that  $|\varphi-\psi|\le \eta+2\delta$. 
\end{proposition}

\begin{proof}
Since $\varphi$ is $\eta$-nondecreasing, if $0\le s\le t$ then $\varphi\left(\tfrac{t-s}{2}\right)\le \varphi\left(\tfrac{t+s}{2}\right)+\eta$ and thus $$\max\left\{\varphi\left(\tfrac{t+s}{2}\right),\varphi\left(\tfrac{t-s}{2}\right)\right\}\le \varphi\left(\tfrac{t+s}{2}\right)+\eta.$$
Since $([0,\infty),|\cdot|_\varphi)$ is $\delta$-hyperbolic, by Proposition~\ref{prop:delta-mt_ineq}, $$\tfrac{1}{2}\varphi(t)+\tfrac{1}{2}\varphi(s)\le \varphi\left(\tfrac{t+s}{2}\right)+\tfrac{1}{2}\eta+\delta,$$ which shows that $\varphi$ is $\left(\tfrac{1}{2}\eta+\delta\right)$-midpoint-concave.  
The existence of a continuous concave metric transform $\psi\in\mathcal{C}$ with $|\varphi-\psi|\le\eta+2\delta$ is given by  Corollary~\ref{cor:appr-midconc_conc}.
\end{proof}

The following result shows that the characterization given by Theorem~\ref{thm:conc_mainA} extends to approximately nondecreasing metric transforms.

\begin{theorem}[Theorem A]\label{thm:appr_nondecr_mainA}
Let $\varphi$ be an approximately nondecreasing, unbounded metric transform.
The transformed metric space $([0,\infty),|\cdot|_\varphi)$ is Gromov hyperbolic if and only if
one of the following two mutually exclusive conditions holds:
\begin{itemize}
\item[$(i)$] $\varphi$ is an approximate dilation, or
\item[$(ii)$] $\varphi$ is logarithm-like. 
\end{itemize}
\end{theorem}

\begin{proof} Fix $\eta, \delta\ge 0$.
Assume $\varphi\in \mathcal{M}$ is an unbounded $\eta$-nondecreasing metric transform such that $([0,\infty),|\cdot|_\varphi)$ is $\delta$-hyperbolic. By Proposition~\ref{prop:incr_hyp_mt}, there exists a continuous concave metric transform $\psi\in \mathcal{C}$ such that $|\varphi(t)-\psi(t)|\le \eta+2\delta$ for all $t\ge 0$. Notice that $\psi$ is an unbounded continuous concave metric transform and the transformed Euclidean half line $([0,\infty),|\cdot|_\psi)$ is $(2\eta+6\delta)$-hyperbolic. \linebreak 
By Theorem~\ref{thm:conc_mainA}, $\psi$ is either an approximate dilation or a logarithm-like metric transform. Since $\varphi$ is within bounded distance from $\psi$, the conclusion follows. 
\end{proof}


\section{Proof of Theorem B}\label{sec:proof_thmB}

In this section we prove Theorem~\ref{thm:B} (Theorem~\ref{appr_nondecr_mainB}) and its corollary (Corollary~\ref{cor:rigidity}) as stated in the introduction.

 Recall that a {\it rough isometric embedding} between two metric spaces $X$ and $Y$ is given by a map $f \colon X\to Y$ and a constant $k\ge0$ such that for all $x,y\in X$
 \[d_X(x,y)-k\le d_Y(f(x),f(y))\le d_X(x,y)+k.\]
\begin{lemma}\label{lem:k-rough_iso}
Assume that $f\colon X\to Y$ is a $k$-rough isometric embedding. If $Y$ is $\delta$-hyperbolic then $X$ is $(\delta+2k)$-hyperbolic.   
\end{lemma}
\begin{proof}
 We use the $4$-point inequality in Proposition~\ref{prop:4points}.  Let $x,y,z,w\in X$. Then
\begin{align*}
d_X(x &,y) + d_X(z,w) \le d_Y(f(x),f(y)) + d_Y(f(z),f(w)) + 2k\\
&\le\max\{d_Y(f(x),f(z)) + d_Y(f(y),f(w)), d_Y(f(x),f(w)) + d_Y(f(y),f(z))\}+2\delta+2k\\
&\le \max\{d_X(x,z)+d_X(y,w)+2k, d_X(x,w)+d_X(y,z)+2k\} +2\delta+2k\\
&=\max\{d_X(x,z)+d_X(y,w), d_X(x,w)+d_X(y,z)\} +2(\delta+2k).
\end{align*}
which shows that $(X,d_X)$ is $(\delta+2k)$-hyperbolic.
\end{proof}

\begin{theorem}\label{thm:main1}
Let $\varphi\in\mathcal{M}$ be an unbounded, approximately nondecreasing metric transform. The transformed Euclidean half line $([0,\infty), |\cdot|_\varphi)$ can be roughly isometrically embedded in a Gromov hyperbolic space $(X,d)$ if and only if $\varphi$ is of one of the following forms:
\begin{itemize}
\item[(i)] $\varphi$ is an approximate dilation,
\item[(ii)] $\varphi$ is logarithm-like.
\end{itemize}
\end{theorem}

\begin{proof}
If $([0,\infty), |\cdot|_\varphi)$ admits a rough isometric embedding into a Gromov hyperbolic space then by Lemma~\ref{lem:k-rough_iso} it is Gromov hyperbolic and the conclusion follows from Theorem~\ref{thm:appr_nondecr_mainA}.
\end{proof}

Recall that a {\it rough geodesic ray} in a metric space $(X,d)$ is a rough isometric embedding of the Euclidean half line $[0,\infty)$ into $X$.

\begin{lemma}\label{lem:appr_nondecr_k-rough_ray}
Let $\gamma\colon[0,\infty)\to (X,d)$ be a $k$-rough geodesic ray and let $\varphi\in\mathcal{M}$ be a $\eta$-nondecreasing, unbounded metric transform. Then $\gamma\colon([0,\infty),|\cdot|_\varphi)\to (X,d_\varphi)$ is a $(\varphi(k)+\eta)$-rough isometric \hbox{embedding.}
\end{lemma}
\begin{proof}
Since $\gamma\colon [0,\infty)\to(X,d)$ is a $k$-rough geodesic ray for all $t,s\ge 0$  \[|t-s|-k\le d(\gamma(t),\gamma(s))\le |t-s|+k\] and since $\varphi$ is $\eta$-nondecreasing \[ \varphi(|t-s|-k)\le \varphi(d(\gamma(t),\gamma(s)))+\eta\le \varphi(|t-s|+k)+2\eta.\] Since $\varphi$ is subadditive \[\varphi(|t-s|)-\varphi(k)\le  \varphi(d(\gamma(t),\gamma(s)))+\eta\le \varphi(|t-s|)+\varphi(k)+2\eta\] or equivalently $|t-s|_\varphi-\varphi(k)-\eta\le d_\varphi(\gamma(t),\gamma(s))\le|t-s|_\varphi+\varphi(k)+\eta$.
\end{proof}

\begin{theorem}[Theorem B]\label{appr_nondecr_mainB}
Let $(X,d)$ be a metric space containing a rough geodesic ray. Let $\varphi$ be an approximately nondecreasing, unbounded metric transform. If the transformed space $(X,d_\varphi)$ is Gromov hyperbolic then
\begin{itemize}
\item[$(i)$] $(X,d)$ is Gromov hyperbolic and $\varphi$ is an approximate dilation, or
\item[$(ii)$] $(X,d_\varphi)$ is approximately ultrametric. \\
Conditions (i) and (ii) are mutually exclusive. 
\end{itemize}
\end{theorem}
\begin{proof}
Let $\gamma\colon[0,\infty)\to X$ be a rough geodesic ray in $(X,d)$.
Then by Lemma~\ref{lem:appr_nondecr_k-rough_ray}, the map 
\linebreak 
$\gamma\colon ([0,\infty),|\cdot|_\varphi)\to (X,d_\varphi)$ is a rough isometric embedding, and since $(X,d_\varphi)$ is Gromov hyperbolic, by Lemma~\ref{lem:k-rough_iso}, the transformed Euclidean half line $([0,\infty),|\cdot|_\varphi)$ is Gromov hyperbolic. 
By Theorem~\ref{thm:appr_nondecr_mainA}, this occurs if and only if the metric transform $\varphi$ is either an approximate dilation or logarithm-like. 
In the former case $(X,d)$ is roughly similar to $(X,d_\varphi)$ by Remark~\ref{rem:appr-dilation_rough-similar} and therefore Gromov hyperbolic by Proposition~\ref{prop:rough-sim_hyp}, and in the latter case $(X,d_\varphi)$ is approximately ultrametric by Proposition~\ref{prop:appr_log-like}.

Suppose that both (i) and (ii) both hold. 
In particular, $\varphi$ is an approximate dilation and hence, because $(X,d_\varphi)$ is approximately ultrametric,  $(X,d)$ must also be approximately ultrametric. 
This is impossible since $(X,d)$ contains a rough geodesic ray. 
\end{proof}

The following corollary of Theorem~\ref{appr_nondecr_mainB} can
be viewed as a type of  rigidity with respect to metric transformation of roughly geodesic Gromov hyperbolic spaces.

\begin{corollary}[Metric Transform Rigidity]\label{cor:rigidity}
Let $(X,d)$ be a metric space containing a rough geodesic ray.
Let $\varphi$ be an approximately nondecreasing, unbounded metric transform.
If the transformed space $(X,d_\varphi)$ is Gromov hyperbolic and roughly geodesic then $\varphi$ is an approximate dilation and $(X,d)$ is Gromov hyperbolic and roughly geodesic.
\end{corollary}
\begin{proof}
Since $(X,d)$ contains a rough geodesic and the transformed space $(X,d_\varphi)$ is Gromov hyperbolic, it follows as in the proof of Theorem~\ref{appr_nondecr_mainB} that $([0,\infty),|\cdot|_\varphi)$ is Gromov hyperbolic and by Theorem~\ref{thm:appr_nondecr_mainA} that $\varphi$ is either an approximate dilation or a logarithm-like metric transform. However, $\varphi$ cannot be a logarithm-like since  Proposition~\ref{prop:appr_log-like} would then imply that $(X,d_\varphi)$ is approximately ultrametric and  by Proposition~\ref{prop:appr_ultra_non_geo} this would contradict the assumption that $(X,d_\varphi)$ is roughly geodesic. Thus $\varphi$ has to be an approximate dilation and by Remark~\ref{rem:appr-dilation_rough-similar} $(X,d)$ is roughly similar to $(X,d_\varphi)$, hence 
the metric space $(X,d)$ is Gromov hyperbolic and roughly geodesic. 
\end{proof}


\bibliographystyle{amsalpha}

\providecommand{\bysame}{\leavevmode\hbox to3em{\hrulefill}\thinspace}
\providecommand{\MR}{\relax\ifhmode\unskip\space\fi MR }
\providecommand{\MRhref}[2]{%
  \href{http://www.ams.org/mathscinet-getitem?mr=#1}{#2}
}
\providecommand{\href}[2]{#2}


\end{document}